\documentclass[12pt]{amsart}
\usepackage{amsmath,amssymb}

\newtheorem{theorem}{Theorem}[section]
\newtheorem{lemma}[theorem]{Lemma}

\newtheorem{proposition}[theorem]{Proposition}

\theoremstyle{definition}

\theoremstyle{remark}
\newtheorem{remark}[theorem]{Remark}

\DeclareMathOperator{\V}{V} \DeclareMathOperator{\osc}{osc}

 \DeclareMathOperator{\Vol}{Vol}

\author[Z.L. Zhang]{Zhenlei Zhang}
\thanks{The author is supported by NSF grant 09221010056 of China}
\address{Department of Mathematics\\
Capital Normal University\\  Xisanhuan North Road 105\\ Beijing,
100048, P.R.China\\}

\email{zhleigo@yahoo.com.cn}


\begin{document}

\title{K\"ahler Ricci flow on Fano manifolds \\with vanished Futaki invariants}

\begin{abstract}
In this paper, we establish several sufficient and necessary
conditions for the convergence of a K\"ahler-Ricci flow, on a
K\"ahler manifold with positive first Chern class, to a
K\"ahler-Einstein metric (or a shrinking K\"ahler-Ricci soliton).
\end{abstract}

\today

\maketitle

\section{Introduction}

Ricci flow has been proved to be an effective method to find
K\"ahler-Einstein metrics on a compact K\"ahler manifold. In
\cite{Ca}, Cao showed the long time existence of the K\"ahler-Ricci
flow and proved that it will converge to a K\"ahler-Einstein metric
when the first Chern class is nonpositive. For a manifold with
positive first Chern class, because of the well-known obstructions
to the existence of constant curvature metric, the convergence of
the K\"ahler-Ricci flow becomes complicated and can hopefully be
characterized by the Hamilton-Tian's conjecture \cite{ChTi2}.
However, suppose extra conditions in the later case, many
convergence results have been proved by experts. Assuming the
existence of K\"ahler-Einstein metric, Chen and Tian \cite{ChTi1,
ChTi2} proved that the K\"ahler-Ricci flow will converge to this
metric if the bisectional curvature is positive; later, Perelman
showed the convergence to this K\"ahler-Einstein metric without any
curvature condition, and this was extended to shrinking
K\"ahler-Ricci solitons by Tian and Zhu \cite{TiZh2}. In
\cite{PSSW1, PSSW2, PhSt}, Phong, Song, Sturm and Weinkove studied
the K\"ahler-Ricci flow with two stability conditions-Mabuchi
K-energy is bounded from below and the positive eigenvalue of
$\bar{\partial}$ operator on $T^{1,0}$ vector fields is uniformly
bounded from below-and proved the convergence to K\"ahler-Einstein
metrics; the result was later generalized to shrinking
K\"ahler-Ricci soliton cases by them \cite{PSSW3}. When considering
the eigenvalue on function space, in \cite{ChLiWa} and \cite{ChLi2},
Chen, Li and Wang proved the convergence result to K\"ahler-Einstein
metrics under some curvature conditions for the initial metric and a
pre-stable condition on the complex structure of the K\"ahler-Ricci
flow. Several stability theorems for K\"ahler-Ricci flow around a
K\"ahler-Einstein metric (or shrinking K\"ahler-Ricci soliton) have
also been proved independently through different methods by Tian and
Zhu \cite{TiZh3}, Zhu \cite{Zhu2}, Chen and Li \cite{ChLi2}, Sun and
Wang \cite{SuWa}, Zheng \cite{Zheng}. Besides these, many other results on the
convergence of K\"ahler-Ricci flow on manifolds with positive first
Chern class, under different conditions, are obtained; see
\cite{CaZh, Ch, ChLi1, ChWa1, ChWa2, ChWa3, MuSz, PhSeSt, RuZhZh,
Se, Sz, To, Zhu1}.

As we have seen, some of the work mentioned above depend more or
less on the eigenvalue estimate of Laplace under the K\"ahler-Ricci
flow; see \cite[\S 10]{ChTi1}, \cite[\S 4-6]{ChLiWa}, \cite[\S
3-6]{ChLi2}, \cite[\S 1-2]{Zhu2}, \cite[\S 3]{SuWa} for details. We
remark that the arguments in these works still rely on a strict
additional condition, namely the K\"ahler-Ricci flow should be close
to a K\"ahler-Einstein metric at any time, or its curvature should
be uniformly bounded, to derive the convergence of the metric.
However, by \cite{PSSW1} and \cite{PSSW3}, when one considers the
convergence of a K\"ahler-Ricci flow whose eigenvalue of
$\bar{\partial}$ operator on vector fields is uniformly bounded away
from zero, these extra conditions can be removed if Mabuchi K-energy
admits a lower bound. Motivated by these observations, in this
paper, we want to study how weak the condition on the eigenvalue is
possible for the smooth convergence of a K\"ahelr-Ricci flow. As a
result, we find several other sufficient and necessary conditions
for the convergence of a K\"ahler-Ricci flow on manifolds with
vanished Futaki invariants; see the theorems in this section and the
Remark \ref{r52} in \S 5.

Now let $(M,g_0)$ be a compact K\"ahler manifold with $c_1(M)>0$.
Suppose its K\"ahler form belongs to $\pi c_1(M)$. Let $g(t)$ be the
solution to the K\"ahler-Ricci flow with initial metric $g(0)=g_0$:
\begin{equation}\label{KRF}
\frac{\partial}{\partial t}g_{i\bar{j}}=-R_{i\bar{j}}+g_{i\bar{j}}.
\end{equation}
By \cite{Ca}, we know that the solution exists for all time $t\in[0,\infty)$. It is obvious that $g(t)$ preserves the K\"ahler class $\pi c_1(M)$ and the volume. Denote by $\V=\int_Mdv_{g_0}$ the volume of this K\"ahler-Ricci flow. Then, by the $\partial\bar{\partial}$ lemma, there exists a family of smooth real-valued functions $u(t)$ such that
\begin{equation}\label{RP}
R_{i\bar{j}}(g(t))+\partial_i\bar{\partial}_ju(t)=g_{i\bar{j}}(t).
\end{equation}
We may suppose that $u(t)$ is normalized such that
\begin{equation}\label{RPN}
\frac{1}{\V}\int_Me^{-u(t)}dv_{g(t)}=1
\end{equation}
where $dv_{g(t)}$ denotes the volume element of the metric $g(t)$. For simplicity, let
\begin{equation}\label{average}
a(t)=\frac{1}{\V}\int_Mu(t)e^{-u(t)}dv_{g(t)}
\end{equation}
denote the average of $u(t)$ with respect to the probability measure $\frac{1}{\V}e^{-u(t)}dv_{g(t)}$.

Recall the Poincar\'e inequality for the K\"ahler manifold
$(M,g(t))$, cf. \cite{Fu},
\begin{equation}\label{poincare}
\int_M|\nabla
f|^2e^{-u(t)}dv_{g(t)}\geq\int_M(f-\bar{f})^2e^{-u(t)}dv_{g(t)},\hspace{0.3cm}\forall
f\in C^\infty(M),
\end{equation}
where $\bar{f}=\frac{1}{\V}\int_Mfe^{-u(t)}dv_{g(t)}$; the equality
holds iff $\nabla f=g^{i\bar{j}}\frac{\partial
f}{\partial\bar{z}^j}\frac{\partial}{\partial z^i}$ is a holomorphic
vector field. Define an operator
$L=-\triangle+g^{i\bar{j}}\nabla_iu\nabla_{\bar{j}}$ for each time
$t$, then the Poincar\'e inequality is equivalent to say that the
first eigenvalue of $L$ is $\geq 1$.

Our first observation is the following theorem, which states that if
the Ricci potential $u(t)$ satisfies a ``strict" Poincar\'e
inequality uniformly, then the K\"ahler-Ricci flow will converge
exponentially fast to a K\"ahler-Einstein metric.

\begin{theorem}\label{t1}
Let $(M,g(t)),t\in[0,\infty)$, be a solution to the K\"ahler-Ricci flow {\rm (\ref{KRF})} whose K\"ahler form lies in $\pi c_1(M)$ and $u(t)$ be associated Ricci potential satisfying {\rm (\ref{RP})} and {\rm (\ref{RPN})}. If
\begin{itemize}
\item[(1.1)] $\int_M|\nabla
u(t)|^2e^{-u(t)}dv_{g(t)}\geq(1+\delta)\int_M(u(t)-a(t))^2e^{-u(t)}dv_{g(t)}$
\end{itemize}
holds for a uniform constant $\delta>0$ independent of $t$, then $g(t)$ converges exponentially fast in the $C^\infty$ sense to a K\"ahler-Einstein metric.
\end{theorem}

\begin{remark}
The condition (1.1) has been partially used in \cite[\S 10]{ChTi1},
\cite[\S 4]{ChLiWa} and \cite[\S 1-2]{Zhu2}, together with some
curvature conditions under the K\"ahler-Ricci flow, to derive the
exponential decay of $\|u-a\|_{C^0}$. We will show in \S 2 that all
these extra conditions are actually unnecessary for the same
purpose.
\end{remark}

\begin{remark}
Although a Poincar\'e type inequality is needed in the proof, we do
not assume the gap estimate on the first eigenvalue of Laplace (or
$L$). All we need is a weaker condition with regard to the Ricci
potential $u$.
\end{remark}

\begin{remark}
In \S 5, we will show that the condition (1.1), is actually also
necessary for the convergence of the K\"ahler-Ricci flow to a
K\"ahler-Einstein metric. 
\end{remark}

Then we can prove the following theorem which relates the
convergence of a K\"ahler-Ricci flow to the estimate of the
``second" eigenvalue of $L$.

\begin{theorem}\label{t2}
Let $(M,g(t)),t\in[0,\infty)$, be a solution to the K\"ahler-Ricci flow {\rm (\ref{KRF})} whose K\"ahler form lies in $\pi c_1(M)$ and $u(t)$ be associated Ricci potential satisfying {\rm (\ref{RP})} and {\rm (\ref{RPN})}. Let $\lambda(t)$ be the smallest eigenvalue, except $1$, of $L$ acting on smooth functions at time $t$. Suppose that
\begin{itemize}
 \item[(2.1)] $M$ has vanished Futaki invariant on $\pi c_1(M)$, and
 \item[(2.2)] $\lambda(t)\geq1+\delta$ for a uniform constant $\delta>0$,
\end{itemize}
then $g(t)$ converges exponentially fast in the $C^\infty$ sense to a K\"ahler-Einstein metric.
\end{theorem}

We shall show that the lower bound of the eigenvalue of
$\bar{\partial}^*\bar{\partial}$ on $T^{1,0}$ vector fields, which
was introduced in \cite{PhSt}, gives rise to an estimate of $\delta$
in Theorem \ref{t2}. As a consequence, we prove the following
generalization of a theorem of Phong, Song, Sturm and Weinkove
\cite{PSSW1}, where the condition of lower bounded Mabuchi K-energy
weakened to the vanishing of the Futaki invariant.

\begin{theorem}\label{t3}
Let $(M,g(t)),t\in[0,\infty)$, be a solution to the K\"ahler-Ricci flow {\rm (\ref{KRF})} whose K\"ahler form lies in $\pi c_1(M)$. Let $\mu(t)$ be the smallest positive eigenvalue of $-g^{i\bar{j}}\nabla_i\nabla_{\bar{j}}$ acting on smooth $T^{1,0}$ vector fields at time $t$. Suppose that
\begin{itemize}
 \item[(3.1)] the Futaki invariant vanishes on $\pi c_1(M)$, and
 \item[(3.2)] $\mu(t)\geq c$ for a uniform constant $c>0$,
\end{itemize}
then $g(t)$ converges exponentially fast in the $C^\infty$ sense to
a K\"ahler-Einstein metric.
\end{theorem}

Finally, we consider the modified K\"ahler-Ricci flow. We adopt the
notion in \cite{PSSW3}. Let $M$ be a compact K\"ahler manifold with
$c_1(M)>0$ and $X$ be a holomorphic vector field whose imaginary
part $Im(X)$ induces an $S^1$ holomorphic action on $M$. Let $g_0$
be an invariant metric whose K\"ahler class is $\pi c_1(M)$. Then
the modified K\"ahler-Ricci flow with respect to the holomorphic
vector field $X$, with initial metric $g(0)=g_0$, is defined by, cf.
\cite{PSSW3}:
\begin{equation}\label{MKRF}
\frac{\partial}{\partial t}g_{i\bar{j}}=-R_{i\bar{j}}+g_{i\bar{j}}+g_{k\bar{j}}\nabla_iX^k.
\end{equation}
This flow preserves the K\"ahler class as well as the invariance by
$S^1$ action. Furthermore, the modified flow is just the
reparametrization of the K\"ahler-Ricci flow (\ref{KRF}) by
holomorphic transformations generated by $Im(X)$. Thus, it still
exists for all time $t\in[0,\infty)$. By a slight modification, we
can prove the similar convergence results for the modified
K\"ahler-Ricci flow; see \S 4 for more details. Here we just state
the following generalization of the main theorem of \cite{PSSW3},
which gives a complete solution to the Question (3) in \cite[\S
8]{PSSW3}.

\begin{theorem}\label{t4}
Let $M$ be a compact K\"ahler manifold with $c_1(M)>0$ and $X$ be a
holomorphic vector field. Let $g(t),t\in[0,\infty)$, be a solution
to the modified K\"ahler-Ricci flow with respect to $X$. Let
$\mu(t)$ be the smallest positive eigenvalue of
$-g^{i\bar{j}}\nabla_i\nabla_{\bar{j}}$ acting on smooth $T^{1,0}$
vector fields at time $t$. Suppose that
\begin{itemize}
 \item[(4.1)] the modified Futaki invariant vanishes on $\pi c_1(M)$, and
 \item[(4.2)] $\mu(t)\geq c$ for a uniform constant $c>0$,
\end{itemize}
then $g(t)$ converges exponentially fast in the $C^\infty$ sense to
a K\"ahler-Ricci soliton with respect to the holomorphic vector
field $X$.
\end{theorem}

We refer to \cite{TiZh1} or \cite{PSSW3} for the definition and
basic properties of the modified Futaki invariant.

The paper is organized as follows: In \S 2, we prove Theorem
\ref{t1}; In \S 3, we consider how the Futaki invariant effects the
Poincar\'e inequality and give a proof of Theorem \ref{t2} and
\ref{t3}; In \S 4, we consider the modified K\"ahler-Ricci flow and
prove Theorem \ref{t4}; In \S 5, we give some further remarks.

\medskip

\noindent {\bf Acknowledgement:} The research was initiated when the author was visiting Professor G. Tian in Princeton University in 2009. He author would like to thank the Department of Mathematics for the hospitality; he would like to thank Professor G. Tian for his helpful discussion and constant help. The author also would like to thank Professors D. H. Phong, J. Song, J. Sturm and B. Weinkove for their interests to this paper, especially Professor D. H. Phong for his suggestions to improve the first version of the paper. Thanks also gives K. Zheng for sending his paper \cite{Zheng}. Finally, The author thanks Y.G. Zhang for his discussion in the course of writing the paper.

\section{K\"ahler-Ricci flow and Poincar\'e inequality}

Let $(M,g_0)$ be a compact K\"ahler manifold of dimension
$\dim_{\mathbb{C}}M=n$ with associated K\"ahler form $\omega_0$.
Suppose $c_1(M)>0$ and $\omega_0\in\pi c_1(M)$. Let
$g(t),t\in[0,\infty)$, be the solution to the K\"ahler-Ricci flow
(\ref{KRF}). Let $\V$ be the volume of the K\"ahler-Ricci flow and
$u(t)$ be the normalized Ricci potentials satisfying (\ref{RP}) and
(\ref{RPN}). Then it is an easy check that $u(t)$ satisfies the
evolution
\begin{eqnarray}
\frac{\partial }{\partial t}u=\triangle u+u-a,
\end{eqnarray}
where $a$ is the average of $u$ defined in (\ref{average}). By
Perelman, cf. \cite{SeTi} for a proof, the Ricci potentials $u(t)$
satisfy the uniform bound under the K\"ahler-Ricci flow:
\begin{equation}\label{Pe}
\|u(t)\|_{C^0}+\|\nabla u(t)\|_{C^0}+\|\triangle u(t)\|_{C^0}\leq C.
\end{equation}
Furthermore, the metrics $g(t)$ are uniformly volume noncollapsed:
\begin{equation}
\Vol_{g(t)}(B_{g(t)}(x,r))\geq\kappa r^{2n},\hspace{0.3cm}\forall
t>0\mbox{ and }r\leq 1.
\end{equation}
Here $C$ and $\kappa$ are uniform positive constants depending only
on $g_0$.

Now we turn to the proof of Theorem \ref{t1}. Several lemmas are
needed.

First of all, by a direct computation, cf. Remark (1) in \cite[\S 6]{PSSW1},
\begin{equation}\label{a dirivative}
\frac{d a}{dt}(t)=\frac{1}{\V}\int_M(|\nabla
u|^2-(u-a)^2)e^{-u}dv_{g(t)}.
\end{equation}
Introduce $$Y=\frac{1}{\V}\int_M(u-a)^2e^{-u}dv,\hspace{0.3cm}
Z=\frac{da}{dt}=\frac{1}{\V}\int_M(|\nabla
u|^2-(u-a)^2)e^{-u}dv_{g}$$ at each time $t$.

\begin{lemma}\label{l21}
For any K\"ahler-Ricci flow we have $Z(t)\rightarrow 0$ as
$t\rightarrow\infty.$
\end{lemma}
\begin{proof}
By Poincar\'e inequality (\ref{poincare}), $Z(t)\geq 0$ for any $t$.
Then observe that
$$\int_0^\infty
Z(t)dt=\lim_{t\rightarrow\infty}a(t)-a(0)<\infty.$$ To show
$Z(t)\rightarrow 0$, it suffices to prove that $\frac{d Z}{dt}$ is
uniformly bounded from above. Recall that the evolution of $|\nabla
u|^2$ is given by, cf. \cite{SeTi},
\begin{equation}\nonumber
\frac{\partial}{\partial t}|\nabla u|^2=\triangle |\nabla
u|^2-|\nabla\nabla u|^2-|\nabla\bar{\nabla}u|^2+|\nabla u|^2.
\end{equation}
Thus, by Perelman's estimate (\ref{Pe}),
\begin{eqnarray}
\frac{d Z}{dt}&=&\frac{d}{dt}\frac{1}{\V}\int_M(|\nabla u|^2-(u-a)^2)e^{-u}dv_{g(t)}\nonumber\\
&=&\frac{1}{\V}\int_M\big[\triangle|\nabla u|^2-|\nabla\nabla u|^2-|\nabla\bar{\nabla}u|^2+3|\nabla u|^2-\triangle(u-a)^2\nonumber\\
&&-2(u-a)^2-(|\nabla u|^2-(u-a)^2)(u-a)\big](2\pi)^{-n}e^{-u}dv\nonumber\\
&=&\frac{1}{\V}\int_M\big[-|\nabla\nabla u|^2-|\nabla\bar{\nabla}u|^2+3|\nabla u|^2-2(u-a)^2\nonumber\\
&&+(|\nabla u|^2-(u-a)^2)(-\triangle u+|\nabla
u|^2-u+a)\big](2\pi)^{-n}e^{-u}dv\nonumber
\end{eqnarray}
is uniformly bounded from above. Here, in the second equality, we used $$\frac{\partial}{\partial t}(e^{-u}dv)=(-\triangle u-u+a+n-s)e^{-u}dv=-(u-a)e^{-u}dv$$ since the Ricci potential satisfies $s+\triangle u=n$.
\end{proof}

\begin{lemma}\label{l22}
Assume as in Theorem \ref{t1}, then $Y(t)\rightarrow 0$ as
$t\rightarrow\infty$.
\end{lemma}
\begin{proof}
Applying the assumption (\ref{t1}) to (\ref{a dirivative}) gives
$$Z=\frac{da}{dt}\geq\frac{\delta}{V}\int_M(u-a)^2e^{-u}dv.$$
Then use above lemma.
\end{proof}

The following lemma gives the $C^0$ estimate of $(u-a)$ in terms of its $L^2$ estimate through Perelman's gradient estimate.

\begin{lemma}\label{l23}
There exists a constant $A$ depending only on $g_0$ such that
\begin{equation}\nonumber
\|u-a\|_{C^0}\leq AY^{\frac{1}{2n+2}}
\end{equation}
at any time.
\end{lemma}
\begin{proof}
Combine equation (4.7) in \cite{PSSW1} with Perelman's $C^0$
estimate of $u$ (\ref{Pe}).
\end{proof}

The following lemma is an essential one in the proof of Theorem
\ref{t1}. We remark that it has appeared in \cite{ChTi1},
\cite{ChTi2} and \cite{ChLiWa} when additional conditions are given.

\begin{lemma}\label{l24}
Assume as in Theorem \ref{t1}, then there exist positive constants
$\gamma$ and $B$ depending on $g(0)$ and $\delta$ such that
\begin{equation}\nonumber
Y(t)\leq  B e^{-\gamma t},\hspace{0.5cm}\forall t\in[0,\infty).
\end{equation}
\end{lemma}
\begin{proof}
By Lemma \ref{l22} and Lemma \ref{l23}, $\|u-a\|_{C^0}\rightarrow 0$
as $t\rightarrow\infty$. Thus,
\begin{eqnarray}
\frac{d}{dt}Y&=&\frac{1}{\V}\int_M\big[2(u-a)(\triangle u+u-a-\frac{da}{dt})-(u-a)^3\big]e^{-u}dv\nonumber\\
&=&\frac{1}{\V}\int_M\big[2(u-a)|\nabla u|^2-2|\nabla u|^2+2(u-a)^2-(u-a)^3\big]e^{-u}dv\nonumber\\
&\leq&\frac{1}{\V}\int_M\big[(-2+2\|u-a\|_{C^0})|\nabla u|^2+(2+\|u-a\|_{C^0})(u-a)^2\big]e^{-u}dv\nonumber\\
&\leq&\big((-2+2\|u-a\|_{C^0})(1+\delta)+(2+\|u-a\|_{C^0})\big)Y\nonumber\\
&\leq&-\delta\cdot Y\nonumber
\end{eqnarray}
whenever $t$ is large enough, where we used the assumption (1.1) in
the second inequality. This suffices to complete the proof of the
lemma.
\end{proof}

Now we are ready to conclude the proof of Theorem \ref{t1}.

\begin{proof}[Proof of Theorem \ref{t1}]
One can use the argument of \cite{PSSW1} to give a proof of the theorem, namely first get the exponential decay of $(s-n)$ by \cite[Lemma 1]{PSSW1} and then apply \cite[Lemma 6]{PSSW1}. Or, adopt another direct proof as follows.

Let $g(t),t\in[0,\infty)$, be the solution to the K\"ahler-Ricci
flow and $u(t)$ be associated Ricci potentials as above. Define a
family of functions
$$\phi(x,t)=\int_0^t(u(x,s)-a(s))ds,\hspace{0,3cm}\forall x\in M,t\geq 0.$$ It is an easy check that
$\phi(t)$ is nothing but the relative K\"ahler potential of $g(t)$
in the sense that
$g_{i\bar{j}}(t)=g_{i\bar{j}}(0)+\partial_i\partial_{\bar{j}}\phi(t)$.
Actually,
$$\partial_i\partial_{\bar{j}}\phi(t)=\int_0^t\partial_i\partial_{\bar{j}}u(s)ds=\int_0^t\frac{\partial g_{i\bar{j}}}{\partial
s}ds=g_{i\bar{j}}(t)-g_{i\bar{j}}(0).$$ Combining with Lemma
\ref{l23} and Lemma \ref{l24} gives the exponential decay of
$\|u-a\|_{C^0}$. Hence, the potentials $\phi(t)$ is uniformly
bounded on $M\times[0,\infty)$. By Yau's resolution of Calabi
conjecture \cite{Ya}, see also \cite{Ti}, the metrics $g(t)$ are all
$C^\infty$ equivalent to each other. By Lemma 1 of \cite{PSSW1},
$\|\nabla u\|_{C^0}$ also decays exponentially, which in turn
implies, by the argument in \cite{PhSt} and \cite{PSSW1}, the
exponential decay of $\|\frac{\partial g}{\partial
t}\|_{C^k}=\|Ric-g\|_{C^k}$ for any $k$. This completes the proof of
the Theorem.
\end{proof}

\section{K\"ahler-Ricci flow and Futaki invariant}

Let $(M,g)$ be any compact K\"ahler manifold whose K\"ahler form
lies in $\pi c_1(M)$ and $u$ be the normalized Ricci potential
satisfying (\ref{RP}) and (\ref{RPN}). Denote by $\V$ its volume.

Denote by $h^0$ the space of holomorphic vector fields on $M$. The
Futaki invariant on the K\"ahler class $\pi c_1(M)$, say
${\tt{F}}:h^0\rightarrow\mathbb{C}$, is defined via \cite{Fu, Ti2}
\begin{equation}\label{Futaki}
{\tt{F}}(X)=\int_MX(u)dv,\hspace{0.3cm}\forall X\in h^0.
\end{equation}
It is well known that $\tt{F}$ does not depend on the specified
chosen metric $g$ in the K\"ahler class $\pi c_1$.

Introduce as in \cite{PSSW2} two inner products on the space of
sections of $T^{1,0}M$:
\begin{equation}\nonumber
\langle V,W\rangle_0=\frac{1}{\V}\int_Mg_{i\bar{j}}V^i\overline{W^j}dv,\hspace{0.3cm}\langle V,W\rangle_u=\frac{1}{\V}\int_Mg_{i\bar{j}}V^i\overline{W^j}e^{-u}dv.
\end{equation}
Denote by $\pi_0$ and $\pi_u$ the orthogonal projections of
$T^{1,0}$ vector fields onto $h^0$ with respect to
$\langle,\rangle_0$ and $\langle,\rangle_u$ respectively. Let
$\nabla u=g^{i\bar{j}}\frac{\partial u}{\partial\bar{z}^j}
\frac{\partial}{\partial z^i}$ be the complex gradient field of $u$.
Then by definition, since $\tt{F}\equiv 0$,
\begin{equation}\nonumber
\langle h^0,\nabla u\rangle_0=0.
\end{equation}
In particular, as observed by Phong and Sturm in \cite{PhSt},
\begin{equation}\nonumber
\int_M|\pi_0(\nabla u)|^2dv={\tt{F}}(\pi_0(\nabla u))=0,
\end{equation}
which implies that $\pi_0(\nabla u)\equiv0$. We first establish the
following proposition which relates the second eigenvalue of
$L=-\triangle+g^{i\bar{j}}\nabla_iu\nabla_{\bar{j}}$ and the strict
Poincar\'e inequality (1.1). The proposition has appeared in
\cite{ChTi1} and \cite{Zhu2} when the K\"ahler metric is close to a
K\"ahler-Einstein metric; in \cite{ChLiWa} and \cite{ChLi2}, a
similar property was proved for manifolds under certain pre-stable
condition.

\begin{proposition}\label{p31}
Let $(M,g)$ be a compact K\"ahler manifold whose K\"ahler form lies in $\pi c_1(M)$. Suppose the Futaki invariant $\tt{F}\equiv 0$ on $\pi c_1(M)$. Then the following general estimate holds
\begin{equation}\label{e31}
\int_M|\nabla u|^2e^{-u}dv\geq(1+\delta^{'})\int_M(u-a)^2e^{-u}dv
\end{equation}
where $a=\frac{1}{\V}\int_Mue^{-u}dv$, while $\delta^{'}$ is a constant depending only on the lower bound of $\delta$ as in Theorem \ref{t2} and the upper bound of $\osc(u)$.
\end{proposition}

Here, $\osc(f)=\max f-\inf f$ for any function $f$. To prove this proposition we need the following lemma.

\begin{lemma}\label{l31}
Let $\nabla u=\pi_u(\nabla u)+V$ be the orthogonal decomposition
with respect to $\langle,\rangle_u$. Then
\begin{equation}\nonumber
\langle\pi_u(\nabla u),\pi_u(\nabla u)\rangle_0\leq\langle
V,V\rangle_0.
\end{equation}
\end{lemma}
\begin{proof}
First observe that
$$0=\langle\pi_u(\nabla u),\nabla u\rangle_0=\langle\pi_u(\nabla u),\pi_u(\nabla u)\rangle_0+\langle\pi_u(\nabla u),V\rangle_0.$$
Then the estimate follows from the Schwarz inequality:
$$\langle\pi_u(\nabla u),\pi_u(\nabla u)\rangle_0=-\langle\pi_u(\nabla u),V\rangle_0\leq\langle\pi_u(\nabla u),\pi_u(\nabla u)\rangle_0^{1/2}\cdot\langle V,V\rangle_0^{1/2}.$$
\end{proof}

Now we give a proof of the proposition.

\begin{proof}[Proof of Proposition \ref{p31}]
For brevity define a probability measure
$d\rho=\frac{1}{\V}e^{-u}dv$ on $M$. Denote by
$\lambda_1<\lambda_2<\cdots$ the sequence of positive eigenvalues of
the operator
$L=-\triangle+g^{i\bar{j}}\nabla_iu\cdot\nabla_{\bar{j}}$ acting on
function space $L^2(d\rho)$ which are bigger than $1$. In view of
the Poincar\'e inequality (\ref{poincare}), the only possible
eigenvalue of $L$ between $0$ and $\lambda_1$ is $1=\lambda_0$ with
engenspace generated by functions whose complex gradient fields are
holomorphic. Let $E_k$ be the eigenspace of $\lambda_k$ and
$$u-a=u_0+u_1+u_2+\cdots$$
be the unique orthogonal decomposition with respect to the measure $d\rho$, where $u_k\in E_k$ and $u_0$ has holomorphic complex gradient field.

For any $k>0$ we have by assumption $\lambda_k\geq\lambda(t)\geq1+\delta$. Thus,
\begin{eqnarray}
\int_M(u-a)^2d\rho&=&\sum_{k=0}^\infty\int_M |u_k|^2d\rho=\sum_{k=0}^\infty\lambda_k^{-1}\int_M|\nabla u_k|^2d\rho\nonumber\\
&\leq&\int_M|\nabla u_0|^2d\rho+\sum_{k=1}^\infty\frac{1}{1+\delta}\int_M|\nabla u_k|^2d\rho\nonumber\\
&=&\int_M(|\nabla u_0|^2+\frac{1}{1+\delta}|V|^2)d\rho\nonumber.
\end{eqnarray}
Notice that $\pi_u(\nabla u)=\nabla u_0$. By Lemma \ref{l31},
$$\int_M|\nabla u_0|^2d\rho\leq e^{-\min u}\langle\nabla u_0,\nabla u_0\rangle_0\leq e^{-\min u}\langle V,V\rangle_0\leq e^{\osc (u)}\int_M|V|^2d\rho.$$
Thus, by a simple calculation,
$$\int_M(u-a)^2d\rho\leq\int_M(|\nabla u_0|^2+\frac{1}{1+\delta}|V|^2)d\rho
\leq\frac{1+e^{\osc(u)}+\delta
e^{\osc(u)}}{(1+\delta)(1+e^{\osc(u)})}\int_M|\nabla u|^2d\rho.$$
Then one can choose $$\delta^{'}=\frac{\delta }{1+e^{\osc(u)}+\delta
e^{\osc(u)}}.$$ The proof of the proposition is complete.
\end{proof}

With the proposition in hand, then Theorem \ref{t2} follows
directly.

Now we turn to prove Theorem \ref{t3}. For this purpose, we need to
reduce the estimate of eigenvalue on vector fields to the estimate
of $\lambda$. Denote by $\mu$ and $\tilde{\mu}$ be the lowest
positive eigenvalue of $\bar{\partial}^*\bar{\partial}$ and
$-g^{i\bar{j}}\nabla_i\nabla_{\bar{j}}+g^{i\bar{j}}\nabla_iu\cdot\nabla_{\bar{j}}$
acting on smooth $T^{1,0}$ vector fields respectively. Then $\mu$
and $\tilde{\mu}$ can be determined as the largest numbers such that
the following hold:
\begin{eqnarray}\nonumber
\int_M|\bar{\nabla} V|^2dv&\geq&\mu\int_M|V|^2dv,\hspace{0.3cm}\forall \langle h^0,V\rangle_0=0,\\
\int_M|\bar{\nabla}
V|^2e^{-u}dv&\geq&\tilde{\mu}\int_M|V|^2e^{-u}dv,\hspace{0.3cm}\forall
\langle h^0,V\rangle_u=0,\nonumber
\end{eqnarray}
The following lemma is essentially due to Phong, Song, Sturm and Weinkove \cite{PSSW2}:

\begin{lemma}\label{l32}
The eigenvalues $\mu$ and $\tilde{\mu}$ are related by
\begin{equation}\nonumber
e^{-\osc(u)}\mu\leq\tilde{\mu}\leq e^{\osc(u)}\mu.
\end{equation}
\end{lemma}
\begin{proof}
Let $V$ be a section of $T^{1,0}M$ such that $\langle V,h^0\rangle_u=0$ and decompose it with respect to $\langle,\rangle_0$ as $V=W+\xi$ such that $\xi\in h^0$ and $\langle W,\xi\rangle_0=0$. Then,
$$0=\langle V,\xi\rangle_u=\langle\xi,\xi\rangle_u+\langle W,\xi\rangle_u.$$
Thus, $\langle\xi,\xi\rangle_u=-\langle W,\xi\rangle_u$. So, $$\langle V,V\rangle_u=\langle V,W\rangle_u=\langle W,W\rangle_u+\langle\xi ,W\rangle_u=\langle W,W\rangle_u
-\langle\xi,\xi\rangle_u\leq\langle W,W\rangle_u.$$ Now, using $\xi\in h^0$,
\begin{eqnarray}
\frac{1}{\V}\int_M|\bar{\nabla}V|^2e^{-u}dv&\geq&e^{-\max(u)}\frac{1}{\V}\int_M|\bar{\nabla}V|^2dv=e^{-\max(u)}\frac{1}{\V}\int_M|\bar{\nabla}W|^2dv\nonumber\\
&\geq&\mu e^{-\max(u)}\langle W,W\rangle_0\geq \mu e^{-\osc(u)}\langle W,W\rangle_u\nonumber\\
&\geq&\mu e^{-\osc(u)}\langle V,V\rangle_u.\nonumber
\end{eqnarray}
In particular, $\tilde{\mu}\geq\mu e^{-osc u}$. A similar argument gives another side estimate.
\end{proof}

Now we are ready to give a proof of Theorem \ref{t3}.

\begin{proof}[Proof of Theorem \ref{t3}]
It suffices to prove the following estimate for any time:
\begin{equation}\label{e33}
\lambda\geq 1+e^{-\osc(u)}\mu.
\end{equation}
Actually this follows from an observation in the proof of the
Poincar\'e Lemma. Let $\psi$ be an eigenfunction of $\lambda$, then
$\langle h^0,\nabla\psi\rangle_u=0$ and thus, cf. Lemma 4.4 in
\cite{PSSW1},
\begin{eqnarray}
(\lambda-1)\int_M|\nabla\psi|^2e^{-u}dv=\int_M|\nabla_i\nabla_j\psi|^2e^{-u}dv
\geq\tilde{\mu}\int_M|\nabla\psi|^2e^{-u}dv\nonumber.
\end{eqnarray}
Then (\ref{e33}) follows directly from Lemma \ref{l32}. The proof of
the theorem is complete.
\end{proof}

\section{Modified K\"ahler-Ricci flow by a holomorphic vector field}

Let $M$ be a compact K\"ahler manifold with $c_1(M)>0$ and $X$ be a
holomorphic vector field whose imaginary part $Im(X)$ induces an
$S^1$ holomorphic action on $M$. Let $g_0$ be an invariant metric
whose K\"ahler class is $\pi c_1(M)$. Let $g(t),t\in[0,\infty)$, be
the solution to the modified K\"ahler-Ricci flow (\ref{MKRF}) and
$\V$ be the volume of this flow. Notice that this flow is just a
reparametrization of the K\"ahler-Ricci flow (\ref{KRF}) starting
from $g_0$, thus Perelman's estimate of the Ricci potential
(\ref{Pe}) remains valid:
\begin{equation}\label{e40}
\|u\|_{C^0}+\|\nabla u\|_{C^0}+\|\triangle u\|_{C^0}\leq C.
\end{equation}
Furthermore, the volume noncollapsing condition holds as well:
\begin{equation}
\Vol_{g(t)}(B_{g(t)}(x,r))\geq\kappa r^{2n},\hspace{0.3cm}\forall
t>0\mbox{ and }r\leq 1.
\end{equation}
Here, $C$ and $\kappa$ are positive constant independent of the time.

By Hodge theory, there exists a family of real-valued functions
$\theta=\theta_X(t)$ such that
\begin{equation}\label{e41}
\partial_{\bar{j}}\theta=g_{i\bar{j}}X^i=X_{\bar{j}}.
\end{equation}
We suppose the functions are normalized by
\begin{equation}\label{e42}
\frac{1}{\V}\int_Me^{\theta}dv\equiv 1.
\end{equation}
Then the modified K\"ahler-Ricci flow (\ref{MKRF}) can be written as \cite{PSSW3}
\begin{equation}\label{MKRF0}
\frac{\partial}{\partial t}g_{i\bar{j}}=-R_{i\bar{j}}+g_{i\bar{j}}+\partial_i\partial_{\bar{j}}\theta.
\end{equation}

As usual, let $u$ be the normalized Ricci potentials determined by
(\ref{RP}) and (\ref{RPN}). For simplicity, define as in \cite{PSSW3} the {\it modified
Ricci potential} $w=u+\theta$. Then $g(t)$ is a stationary solution
of the modified K\"ahler-Ricci flow iff $w$ is constant. In \cite{PSSW3}, the authors also proved similar estimates for $w$ and $\theta$:
\begin{eqnarray}\label{e43}
\|\nabla\theta\|_{C^0}+\|\triangle\theta\|_{C^0}+\|w\|_{C^0}+\|\nabla w\|_{C^0}+\|\triangle w\|_{C^0}\leq C
\end{eqnarray}
for some $C$ independent of $t$. In particular, $\|X\|_{C^0}$ is uniformly bounded.

Under the normalizing condition (\ref{RPN}) and (\ref{e42}), $\theta$ and $w$ satisfy
\begin{eqnarray}\label{e44}
\frac{\partial\theta}{\partial t}&=&X(w),\\
\frac{\partial w}{\partial t}&=&(\triangle+X)w+w-a_X,\label{e45}
\end{eqnarray}
where $a_X=\frac{1}{\V}\int_Mwe^{-u}dv$ is the average of $w$. Obviously $a_X$ is uniformly bounded under the modified K\"ahler-Ricci flow.

Notice that the modified K\"ahler-Ricci flow can be written as $\frac{\partial}{\partial t}g_{i\bar{j}}=\partial_i\partial_{\bar{j}}w$. Then following the arguments in \S 2, just by replacing the Ricci potential $u$ by the modified Ricci potential $w$, one can prove the following theorem:

\begin{theorem}\label{t52}
Let $g(t),t\in[0,\infty)$, be a solution to modified K\"ahler-Ricci
flow (\ref{MKRF}) on a compact K\"ahler manifold $M$ with positive
$c_1(M)$. Suppose the modified Ricci potential $w(t)$ satisfies
\begin{equation}\label{e46}
\int_M|\nabla w|^2e^{-u}dv\geq(1+\delta)\int_M(w-a_X)^2e^{-u}dv
\end{equation}
for some $\delta>0$ independent of $t$, then $g(t)$ converges exponentially fast in the $C^\infty$ sense to a K\"ahler-Ricci soliton.
\end{theorem}

Here we mention that to show the exponential decay of $\|\nabla w\|_{C^0}$ one should use the Smoothing Lemma 10 in \cite{PSSW3}.

For application, recall the definition of the modified Futaki
invariant with respect to the holomorphic vector field \cite{TiZh1},
${\tt{F}}_X:h^0\rightarrow\mathbb{C}$,
\begin{equation}\label{modified Futaki}
{\tt{F}}_X(Y)=\int_MY(w)e^{\theta}dv
\end{equation}
where $w$ is the modified Ricci potential and $\theta$ is defined
via (\ref{e41}) and (\ref{e42}). It is showed in \cite{TiZh1} that
the definition does not depend on the specified metric $g$ in $\pi
c_1(M)$.

Let $\lambda$ be the first eigenvalue, except $1$, of
$L=-\triangle+g^{i\bar{j}}\nabla_iu\nabla_{\bar{j}}$ acting on
function space. As in \S 3, we have the following general
proposition:

\begin{proposition}\label{p41}
Let $(M,g)$ be any compact K\"ahler manifold whose K\"ahler form lies in $\pi c_1(M)$. Suppose the modified Futaki invariant ${\tt{F}}_X\equiv 0$ on $\pi c_1(M)$. Denote $\lambda=1+\delta$, then the following general estimate holds
\begin{equation}\label{e47}
\int_M|\nabla w|^2e^{-u}dv\geq(1+\delta^{'})\int_M(w-a_X)^2e^{-u}dv
\end{equation}
where $\delta^{'}$ is a constant depending only on the lower bound of $\delta$ and the upper bound of $\osc(w)$.
\end{proposition}
\begin{proof}
Introduce two inner products on the space of sections of $T^{1,0}M$:
\begin{equation}\nonumber
\langle V,W\rangle_\theta=\frac{1}{\V}\int_Mg_{i\bar{j}}V^i\overline{W^j}e^\theta dv,\hspace{0.3cm}\langle V,W\rangle_u=\frac{1}{\V}\int_Mg_{i\bar{j}}V^i\overline{W^j}e^{-u}dv.
\end{equation}
Define a probability measure $d\rho=\frac{1}{\V}e^{-u}dv$ and let
$w-a_X=w_0+w_+$ be an orthogonal decomposition with respect to
$d\rho$ where $w_0$ has holomorphic complex gradient. In other
words, $w_0$ is nothing but the projection of $w$ onto the
eigenspace of $L$ with eigenvalue $1$. Then as in the proof of
Proposition \ref{p31}, one can show that
\begin{equation}\nonumber
\int_M(w-a_X)^2d\rho\leq\int_M(|\nabla
w_0|^2+\frac{1}{1+\delta}|\nabla w_+|^2)d\rho.
\end{equation}
Next we claim
\begin{equation}\label{e48}
\int_M|\nabla w_0|^2d\rho\leq e^{\osc(w)}\int_M|\nabla w_+|^2d\rho.
\end{equation}
Suppose the inequality for the moment, then
$$\int_M(w-a_X)^2d\rho\leq\frac{1+e^{\osc(w)}+\delta e^{\osc(w)}}{(1+\delta)(1+e^{\osc(w)})}\int_M|\nabla w|^2d\rho$$
and thus one can choose $\delta^{'}=\frac{\delta}{1+e^{\osc(w)}+\delta e^{\osc(w)}}$.

We finally prove (\ref{e48}). Denote by $\pi_\theta$ the projection onto $h^0$ with respect to $\langle,\rangle_\theta$. Since by assumption
$${\tt{F}}_X(\pi_\theta(\nabla w))=\langle\pi_\theta(\nabla w),\nabla w\rangle_\theta=\langle\pi_\theta(\nabla w),\pi_\theta(\nabla w)\rangle_\theta=0,$$ we have $\pi_\theta(\nabla w)=0$. Then one can derive as in Lemma \ref{l32} that
$$\langle\nabla w_0,\nabla w_0\rangle_\theta\leq\langle\nabla w_+,\nabla w_+\rangle_\theta.$$
Now,
$$\langle\nabla w_0,\nabla w_0\rangle_u\leq e^{-\min w}\langle\nabla w_0,\nabla w_0\rangle_\theta\leq e^{-\min w}\langle\nabla w_+,\nabla w_+\rangle_\theta\leq e^{\osc(w)}\langle\nabla w_+,\nabla w_+\rangle_u$$
as claimed in (\ref{e48}). This completes the proof of the theorem.
\end{proof}

As a corollary, one has

\begin{theorem}
Let $(M,g(t)),t\in[0,\infty)$, be a solution to the modified K\"ahler-Ricci flow {\rm (\ref{MKRF})} whose K\"ahler form lies in $\pi c_1(M)$ and $u(t)$ be associated Ricci potential satisfying {\rm (\ref{RP})} and {\rm (\ref{RPN})}. Let $\lambda(t)$ be the smallest eigenvalue, except $1$, of $L$ acting on smooth functions at time $t$. Suppose that
\begin{itemize}
 \item[(1)] $M$ has vanished modified Futaki invariant on $\pi c_1(M)$, and
 \item[(2)] $\lambda(t)\geq1+\delta$ for a uniform constant $\delta>0$,
\end{itemize}
then $g(t)$ converges exponentially fast in the $C^\infty$ sense to
a K\"ahler-Ricci soliton with respect to the holomorphic vector
field $X$.
\end{theorem}

Combining with the estimate (\ref{e33}), Theorem \ref{t4} follows
directly.

\section{Further remarks}

\begin{remark}\label{r51}
The average $a$ (respectively $a_X$) in the K\"ahler-Ricci flow
(respectively modified K\"ahler-Ricci flow) can be bounded from
above by Jensen inequality as follows:
\begin{equation}\nonumber
a=\frac{1}{\V}\int_Mue^{-u}dv\leq\ln(\frac{1}{\V}\int_M dv)=0,
\end{equation}
\begin{equation}\nonumber
a_X=\frac{1}{\V}\int_Mwe^{-u}dv\leq\ln(\frac{1}{\V}\int_M e^{w-u}dv)=\ln(\frac{1}{\V}\int_Me^{\theta}dv)=0.
\end{equation}
Combining with the monotonicity of $a$ (respectively $a_X$) under
the K\"ahler-Ricci flow (respectively modified K\"ahler-Ricci flow),
we have the two-sided bound:
\begin{equation}\label{e51}
a(0)\leq a(t)\leq 0,\hspace{0.3cm}a_X(0)\leq a_X(t)\leq 0,\hspace{0.3cm}\forall t.
\end{equation}
Thus, $a$ (respectively $a_X$) admits a natural bound independent of the $C^0$ bound of $u$ (respectively $w$) in a prior. The monotonicity of $a$ is proved in \cite[Remark (1)]{PSSW1}; the monotonicity of $a_X$ follows from a similar calculation:
\begin{equation}\label{e47}
\frac{d a_X}{dt}=\frac{1}{\V}\int_M\big(|\nabla
w|^2-(w-a_X)^2\big)e^{-u}dv
\end{equation}
which is nonnegative by Poincar\'e inequality. Furthermore, the increasing of $a_X$ is strict unless $g(t)$ is a shrinking K\"ahler-Ricci soliton.
\end{remark}

\begin{remark}\label{r52}
We remark that all the conditions in our theorems on the convergence
of the K\"ahler-Ricci flow are not only sufficient, but also
necessary. Indeed, as shown in \cite{PSSW1}, the convergence of the
K\"ahler-Ricci flow implies the condition (3.1), while the existence
of Einstein metric implies the vanishing of the Futaki invariant.
The conditions (2.1) and (1.1) follow from condition (3.1) when the
Futaki invariant vanishes, as shown in \S 3.
\end{remark}

\begin{remark}
The method presented in this paper can also be applied to other
problems, such as Questions (1.2) and (1.3) in \cite{ChLi2}, i.e.,
the convergence of the K\"ahler-Ricci flow with an ``almost
Einstein" initial metric and the stability of the K\"ahler-Ricci
flow around a K\"ahler-Einstein metric (or more generally a
shrinking K\"ahler-Ricci soliton). Here, ``almost Einstein" means
that the initial metric is close to an Einstein metric in certain
sense. The advantage in our argument is the absence of the
pre-stable condition as well as the curvature bounding condition in
a prior. We will get back to this point in the future.
\end{remark}

\end{document}